\documentclass[11pt]{article}
\usepackage{amscd, amsmath, amssymb, amsthm}
\usepackage[all,cmtip]{xy}
\usepackage[pagebackref]{hyperref}
\usepackage{graphicx}
\DeclareMathAlphabet{\mathpzc}{OT1}{pzc}{m}{it}

\usepackage{tikz}
\usepackage{pgffor}
\usetikzlibrary{arrows,backgrounds}
\usetikzlibrary{calc}

\title{Bott vanishing for algebraic surfaces}
\author{Burt Totaro}
\date{  }

\def\Z{\text{\bf Z}}

\def\C{\text{\bf C}}
\def\P{\text{\bf P}}

\def\arrow{\rightarrow}

\DeclareMathOperator{\Aut}{Aut}
\DeclareMathOperator{\ch}{ch}
\DeclareMathOperator{\td}{td}
\DeclareMathOperator{\Pic}{Pic}

\DeclareMathOperator{\corank}{corank}

\DeclareMathOperator{\Cliff}{Cliff}

\newcommand{\Mod}[1]{\ (\mathrm{mod}\ #1)}
\def\p{\mathcal{P}}
\def\m{\mathcal{M}}

\begin{document}
\maketitle
\newtheorem{theorem}{Theorem}[section]
\newtheorem{corollary}[theorem]{Corollary}
\newtheorem{lemma}[theorem]{Lemma}

\theoremstyle{definition}
\newtheorem{definition}[theorem]{Definition}
\newtheorem{example}[theorem]{Example}

\theoremstyle{remark}
\newtheorem{remark}[theorem]{Remark}

\begin{center}
{\it For William Fulton on his eightieth birthday}
\end{center}

A smooth projective variety $X$ over a field is said to satisfy
{\it Bott vanishing }if
$$H^j(X,\Omega^i_X\otimes L)=0$$
for all ample line bundles $L$, all $i\geq 0$, and all $j>0$. Bott proved
this when $X$ is projective space.
Danilov and Steenbrink extended Bott vanishing to
all smooth projective toric varieties; proofs can be found
in \cite{BC, BTLM, Mustata, Fujino}.

What does Bott vanishing mean? It does not have a clear
geometric interpretation in terms of the classification
of algebraic varieties. But it is useful when it holds,
as a sort of preprocessing step, since
the vanishing of higher cohomology lets us compute
the spaces of sections of various important vector bundles.
Bott vanishing includes Kodaira vanishing
as a special case (where $i$ equals $n:=\dim X$), but it says much more.

For example, any Fano variety
that satisfies Bott vanishing
must be rigid, since $H^1(X,TX)=H^1(X,\Omega^{n-1}_X
\otimes K_X^*)=0$ for $X$ Fano.
So Bott vanishing holds for only finitely many
smooth complex Fano varieties in each dimension. Even among rigid Fano
varieties, Bott vanishing fails for quadrics of dimension at least 3
and for Grassmannians other than projective space \cite[section 4]{BTLM}.
As a result, Achinger, Witaszek,
and Zdanowicz asked whether a rationally connected variety
that satisfies Bott vanishing must be a toric variety
\cite[after Theorem 4]{AWZ}.

In this paper, we exhibit several new classes of varieties
that satisfy Bott vanishing.
First, we answer Achinger-Witaszek-Zdanowicz's question:
there are non-toric rationally connected varieties
that satisfy Bott vanishing, since {\it Bott vanishing holds
for the quintic del Pezzo surface }(Theorem \ref{quintic}).
Over an algebraically closed field,
a quintic del Pezzo surface is
isomorphic to the moduli space $\overline{M_{0,5}}$ of 5-pointed
stable curves of genus zero. It is the only rigid del Pezzo surface
that is not toric: del Pezzo surfaces of degree at least 5 are rigid,
and those of degree at least 6 are toric.
(The quintic del Pezzo surface
also does not have a lift of the Frobenius endomorphism
from $\Z/p$ to $\Z/p^2$, a property known to imply
Bott vanishing \cite{BTLM}, \cite[Proposition 7.1.4]{AWZ}.)
In view of this example, there is a good hope
of finding more Fano or rationally connected varieties that satisfy
Bott vanishing.

We also consider varieties that are not
rationally connected, with most of the paper devoted
to K3 surfaces. Bott vanishing holds for abelian varieties
over any field: it reduces to Kodaira vanishing, since the tangent
bundle is trivial. On the other hand, Riemann-Roch shows that
Bott vanishing fails for all K3 surfaces of degree less than 20
(Theorem \ref{under20}).
But recent work
of Ciliberto-Dedieu-Sernesi and Feyzbakhsh \cite{CDS, Feyzbakhsh} implies:
{\it Bott vanishing holds
for all K3 surfaces of degree 20
or at least 24 with Picard number 1} (Theorem \ref{lowdegree20}).
Version 2 of this paper on the arXiv gave a more elementary
proof, not using Feyzbakhsh's work on Mukai's program (reconstructing
a K3 surface from a curve),
but here we give a short proof using her work.
Surprisingly, Bott vanishing
fails in degree 22.

More strongly, we end up with a clear geometric understanding
of the meaning of Bott vanishing
for a K3 surface with any Picard number; see Theorems \ref{lowdegree74},
\ref{degree1}, and \ref{degree4}.
The key question is whether
$H^1(X,\Omega^1_X\otimes B)$ is zero for an ample line bundle $B$.
This cohomology group has a direct geometric meaning, related
to the map from the moduli space of curves on K3 surfaces
to the moduli space of curves (section \ref{bott20}).

Roughly speaking,
the failure of this vanishing for a K3 surface is caused
{\it either }by elliptic curves of low degree on the surface,
{\it or }by the existence of a (possibly singular) Fano 3-fold
in which the K3 surface is a hyperplane section.
The proofs build on a long development, starting with
the work of Beauville, Mori, and Mukai
about moduli spaces of K3 surfaces, and leading up to recent advances
by Arbarello-Bruno-Sernesi and Ciliberto-Dedieu-Sernesi
\cite{Beauville, MM, Mukai, ABShyperplane, CDS}. We give a complete
description of all K3 surfaces $X$
with an ample line bundle $B$ of high degree
such that
$H^1(X,\Omega^1_X\otimes B)$ is not zero. The most novel aspect
of the paper is our analysis of what happens when there is an elliptic curve
of low degree (Theorem \ref{degree4}). (In other terminology,
this concerns K3 surfaces that are unigonal, hyperelliptic,
trigonal, or tetragonal.) It turns out that
the crucial issue is whether an elliptic fibration has a certain special
type of singular fiber.

I thank Ben Bakker, John Ottem, Zhiyu Tian, and a referee
for proposing important steps in the paper.
I also thank Valery Alexeev,
Enrico Arbarello, William Baker, Daniel Huybrechts,
Emanuele Macr\`i, Scott Nollet,
Kieran O'Grady, and Mihnea Popa
for their suggestions.
This work was supported by National Science Foundation
grant DMS-1701237, and by grant DMS-1440140
while the author was in residence at the
Mathematical Sciences Research Institute in Berkeley, California, during the
Spring 2019 semester.

\section{Notation}

We take a {\it variety} over a field $k$ to mean an integral separated
scheme of finite type over $k$. A {\it curve }means a variety
of dimension 1. So, in particular, a curve is irreducible.
A property is said to hold
for {\it general }(resp.\ {\it very general}) complex points of a variety $Y$
if it holds outside a finite (resp.\ countable) union of closed subvarieties
not equal to $Y$.

On a smooth variety, we often identify line bundles
with divisors modulo linear equivalence.
For example, the tensor product $A\otimes B$
of two line bundles may also be written as $A+B$.
A line bundle is {\it primitive }if it cannot be written as
a positive integer at least 2 times some line bundle.

\section{Bott vanishing for the quintic del Pezzo surface}

\begin{theorem}
\label{quintic}
Let $X$ be a del Pezzo surface of degree 5 over a field $k$.
Then $X$ satisfies Bott vanishing, but is not toric.
\end{theorem}

\begin{proof}
It suffices to prove the theorem after extending $k$,
and so we can assume that $k$ is algebraically closed.
In this case, there is a unique del Pezzo surface $X$
(a smooth projective surface with ample anticanonical bundle $K_X^*$)
of degree 5 over $k$, up to isomorphism. It can be described
as the blow-up of $\P^2$ at any set of 4 points with no three
on a line \cite[Remark 24.4.1]{Manin}.
Here $X$ has finite automorphism group, because
any automorphism of $X$ in the identity component of $\Aut(X)$
would pass to an automorphism of $\P^2$ (that is, an element
of $PGL(3,k)$) that fixes the 4 chosen points, and such
an automorphism must be the identity. In particular, $X$
is not a toric variety. (In fact, the automorphism group of $X$
is the symmetric group $S_5$, but we will not use that.)

The Picard group of $X$ is isomorphic to $\Z^5$, and so Bott
vanishing must be checked for a fairly large (infinite) class
of ample line bundles. We argue as follows. Recall the
Kodaira-Akizuki-Nakano vanishing theorem \cite[Theorem 4.2.3]{Lazarsfeld},
\cite{DI}:

\begin{theorem}
\label{kan}
(1) Every smooth projective variety over a field of characteristic
zero satisfies {\it Kodaira-Akizuki-Nakano vanishing}:
$$H^j(X,\Omega^i\otimes L)=0$$
for all ample line bundles $L$ and all $i+j>\dim(X)$.

(2) Let $X$ be a smooth projective variety over a perfect
field of characteristic $p>0$. If $X$ lifts to $W_2(k)$
and $X$ has dimension $\leq p$, then $X$ satisfies KAN vanishing
(as in (1)).
\end{theorem}

It follows that the quintic del Pezzo surface $X$ satisfies
KAN vanishing: the hypotheses of (2) hold if $k$
has characteristic $p$.
Thus we know that $H^j(X,\Omega^2\otimes L)=0$ 
for all ample
line bundles $L$ and all $j>0$. Since $K_X^*=(\Omega^2_X)^*$ is ample,
it follows that $H^j(X,L)=0$ for ample $L$ and $j>0$. Also by KAN vanishing,
we have $H^2(X,\Omega^1\otimes L)=0$ for ample $L$. To prove
Bott vanishing, it remains
to show that $H^1(X,\Omega^1\otimes L)=0$ for ample $L$.

For any del Pezzo surface $X$ of degree at most 7,
the cone of curves is spanned by the finitely many lines in $X$
(or equivalently, $(-1)$-curves, meaning curves $C$ in $X$ isomorphic
to $\P^1$ with $C^2=-1$; then $(-K_X)\cdot C = 1$)
\cite[section 6.5]{Debarre}.
Therefore, a line bundle $L$ on $X$ is nef if and only if
it has nonnegative degree on all $(-1)$-curves in $X$, and it is ample
if and only if it has positive degree on all $(-1)$-curves in $X$.

We return to the del Pezzo surface $X$ of degree 5 (in which
case there are 10 $(-1)$-curves, shown in Figure \ref{petersen}).
Let $L$ be any ample
line bundle on $X$, and let $a$ be the minimum degree of $L$
on the $(-1)$-curves, which is a positive integer.
Since $-K_X$ has degree 1 on
each $(-1)$-curve, $L$ can be written
(using additive notation for line bundles) as
$$L=a(-K_X)+M$$
for some nef line bundle $M$ on $X$ which has degree zero
on some $(-1)$-curve.

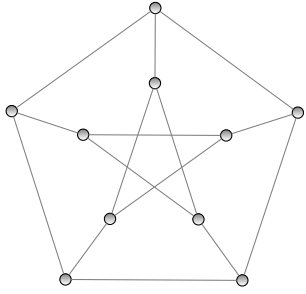
\begin{figure}\centering
\begin{tikzpicture}[yscale=0.5,xscale=0.5,rotate=17.7]
\begin{scope} [vertex style/.style={draw,
                                       circle,
                                       minimum size=1.5mm,
                                       inner sep=0pt,
                                       outer sep=0pt,
                                       shade}] 
      \path \foreach \i in {0,...,4}{%
       (72*\i:2) coordinate[vertex style] (a\i)
       (72*\i:4) coordinate[vertex style] (b\i)}
       ; 
    \end{scope}

     \begin{scope} [edge style/.style={draw=gray}]
       \foreach \i  in {0,...,4}{%
       \pgfmathtruncatemacro{\nextb}{mod(\i+1,5)}
       \pgfmathtruncatemacro{\nexta}{mod(\i+2,5)} 
       \draw[edge style] (a\i)--(b\i);
       \draw[edge style] (a\i)--(a\nexta);
       \draw[edge style] (b\i)--(b\nextb);
       }  
     \end{scope}
  \end{tikzpicture}
\caption{Dual graph of the 10 $(-1)$-curves on the quintic del Pezzo surface}
\label{petersen}
\end{figure}

Choose a $(-1)$-curve $E$ on which $M$ has degree zero,
and let $Y$ be the smooth projective surface obtained
by contracting $E$ (by Castelnuovo's contraction theorem).
Then $Y$ is a del Pezzo surface of degree 6, and such a surface
is toric. Since $M$ has degree 0 on $E$, the isomorphism
$\Pic(X)=\Pic(Y)\oplus \Z$ for a blow-up implies that
$M$ is pulled back from a line bundle on $Y$, which we also call $M$.
Clearly $M$ is nef on $Y$. By Bott vanishing on $Y$, we have
$$H^1(Y,\Omega^1\otimes K_Y^*\otimes M)=0,$$
using that $K_Y^*$ is ample. Here $\Omega^1_Y\otimes K_Y^*\cong 
TY$ (as on any surface), and so
$$H^1(Y,TY\otimes M)=0.$$

For any blow-up $\pi\colon X\arrow Y$ of a point $y$ on a smooth surface $Y$,
we have $R\pi_*(TX)=\pi_*(TX)=TY\otimes I_{y/Y},$ where $I_{y/Y}$ is the
ideal sheaf of $y$ in $Y$. (That is, vector fields on $X$ are equivalent
to vector fields on $Y$ that vanish at $y$.) We have an exact
sequence of coherent sheaves on $Y$,
$$0\arrow I_{y/Y}\arrow O_Y\arrow O_y\arrow 0.$$
Tensoring with the vector bundle $TY\otimes M$ gives another exact sequence,
$$0\arrow TY\otimes M\otimes I_{y/Y}\arrow TY\otimes M\arrow 
(TY\otimes M)|_y\arrow 0.$$
Combining this with the isomorphism above gives a long
exact sequence of cohomology:
$$H^0(Y,TY\otimes M)\arrow (TY\otimes M)|_y\arrow
H^1(X,TX\otimes \pi^*(M))\arrow H^1(Y,TY\otimes M).$$
Here $H^1(Y,TY\otimes M)=0$ by Bott vanishing as above.
Therefore, to show that $H^1(X,TX\otimes\pi^*(M))=0$,
it suffices to show that the rank-2 vector bundle $TY\otimes M$
is spanned at the point $y$ by its global sections. This follows
if we can show that $TY$ and $M$ are spanned at $y$ by their
global sections. For $TY$, this is clear by the vector fields
coming from the action of the torus $T=(G_m)^2$
on $Y$, since $y$ must be in the open $T$-orbit. (The blow-up of $Y$
at a point not in the open $T$-orbit would contain a $(-2)$-curve
and hence could not be a del Pezzo surface.) Also, every nef line
bundle $M$ on a toric variety $Y$ is basepoint-free
\cite[section 3.4]{Fulton}.
Thus we have shown that $H^1(X,TX\otimes \pi^*(M))=0$.

To prove Bott vanishing for $X$, as discussed above, we have
to show that $H^1(X,\Omega^1\otimes (K_X^*)^{\otimes a}\otimes
\pi^*(M))=0$ for all positive integers $a$. Equivalently,
we want $H^1(X,TX\otimes (K_X^*)^{\otimes a-1}\otimes \pi^*(M))=0$
for all positive integers $a$. We have proved this for $a=1$.
By induction, suppose we know this statement for $a$, and then
we will show that $H^1(X,TX\otimes (K_X^*)^{\otimes a}\otimes \pi^*(M))=0$.

On $X$ (as on any del Pezzo surface of degree at least 3),
the line bundle $K_X^*$ is very ample, and so it has a section
whose zero locus is a smooth curve $C$. By the adjunction formula,
$K_C$ is trivial; that is, $C$ has genus 1. We have an exact sequence
$$0\arrow O_X(-C)\arrow O_X\arrow O_C\arrow 0$$
of coherent sheaves on $X$, where $O(-C)\cong K_X$.
Tensoring with the vector bundle
$TX\otimes (K_X^*)^{\otimes a}\otimes\pi^*(M)$ gives another
exact sequence of sheaves, and hence a long exact sequence
of cohomology:
\begin{multline*}
H^1(X,TX\otimes (K_X^*)^{\otimes a-1}\otimes\pi^*(M))
\arrow H^1(X,TX\otimes (K_X^*)^{\otimes a}\otimes\pi^*(M))\\
\arrow H^1(C,(TX\otimes (K_X^*)^{\otimes a}\otimes\pi^*(M))|_C).
\end{multline*}

By induction, the first group shown is zero. Also, the restriction
of $TX$ to $C$ is an extension
$$0\arrow TC\arrow TX|_C\arrow N_{C/X}\arrow 0,$$
where $N_{C/X}\cong (K_X^*)|_C$ by definition of $C$.
Since $C$ has genus 1, this says that the restriction of $TX$
to $C$ is an extension of two line bundles of nonnegative degree.
Since $K_X^*$ is ample on $X$, $\pi^*(M)$ is nef, and $a$ is positive,
it follows that $TX\otimes (K_X^*)^{\otimes a}\otimes
\pi^*(M)$ restricted to $C$ is an extension of two
line bundles of positive degree. Since $C$ has genus 1, $H^1$ of every
line bundle of positive degree on $C$ is zero. We conclude
that the group on the right of the exact sequence above is
zero (like the group on the left). Therefore,
$$H^1(X,TX\otimes (K_X^*)^{\otimes a}\otimes\pi^*(M))=0,$$
which completes the induction. We have shown that
$X$ satisfies Bott vanishing.
\end{proof}

There are also higher-dimensional Fano varieties which satisfy
Bott vanishing but are not toric, in view of:

\begin{lemma}
Let $X$ and $Y$ be smooth projective varieties
over an algebraically closed field. Suppose
that $H^1(X,O)=0$. If $X$ and $Y$ satisfy Bott vanishing,
then so does $X\times Y$.
\end{lemma}

\begin{proof}
Since $H^1(X,O)=0$, we have $\Pic(X\times Y)=\Pic(X)\oplus \Pic(Y)$
\cite[exercise III.12.6]{Hartshorne}. That is, every line bundle
on $X\times Y$ has the form $\pi_1^*L\otimes \pi_2^*M$ for some
line bundles $L$ on $X$ and $M$ on $Y$, where $\pi_1$ and $\pi_2$
are the two projections of $X\times Y$.

By the K\"unneth formula \cite[Tag 0BEC]{Stacks},
$$H^0(X\times Y,\pi_1^*L\otimes \pi_2^*M)
\cong H^0(X,L)\otimes_k H^0(Y,M).$$
Therefore, $\pi_1^*L\otimes \pi_2^*M$ is very ample
on $X\times Y$ if and only if $L$ and $M$ are very ample. It follows
that $\pi_1^*L\otimes \pi_2^*M$ is ample on $X\times Y$
if and only if $L$ and $M$ are ample.

Assume that $X$ and $Y$ satisfy Bott vanishing. We need to show
that 
$$H^j(X\times Y,\Omega^i_{X\times Y}\otimes \pi_1^*L\otimes
\pi_2^*M)=0$$
for all $j>0$, $i\geq 0$, and $L$ and $M$ ample line
bundles. Here $\Omega^i_{X\times Y}=\oplus_m \pi_1^*\Omega^m_X\otimes 
\pi_2^*\Omega^{i-m}_Y$. So the desired vanishing follows from
Bott vanishing on $X$ and $Y$ using the K\"unneth formula.
\end{proof}

\begin{remark}
The Kawamata-Viehweg vanishing theorem extends Kodaira
vanishing to nef and big line bundles,
but it seems unreasonable to ask when Bott vanishing
holds for nef and big line bundles. Indeed, Bott vanishing
fails for a nef and big line bundle on the blow-up
of $\P^2$ at a point, which is about as simple as you can get.
Even KAN vanishing fails for a nef and big line bundle
on the blow-up of $\P^3$ at a point
\cite[Example 4.3.4]{Lazarsfeld}.
\end{remark}

\section{Bott vanishing for K3 surfaces of Picard number 1}
\label{bott20}

We now show that Bott vanishing fails for K3 surfaces of degree
less than 20 or equal to 22,
while it holds for all K3 surfaces
of degree 20 or at least 24 with Picard number 1.
We give a quick proof by applying recent work
of Ciliberto-Dedieu-Sernesi, Arbarello-Bruno-Sernesi,
and Feyzbakhsh, which we discuss
in more detail in section \ref{74}.
In later sections, we consider what happens for K3 surfaces
with Picard number greater than 1.

The less precise statement that Bott vanishing holds
for very general K3 surfaces of degree 20 or at least 24
follows from the work of Beauville, Mori, and Mukai
in the 1980s \cite[section 5.2]{Beauville},
\cite[Theorem 1]{MM}, \cite[Theorem 7]{Mukai}.

Note that Ciliberto, Dedieu, Galati, and Knutsen recently proved
the analog of Beauville-Mori-Mukai's result for Enriques surfaces,
in particular computing $H^1(X,\Omega^1\otimes B)$
for $(X,B)$ a general member
of any component of the moduli space of polarized Enriques surfaces
\cite{CDGK}. By analogy with the results in this paper for K3 surfaces,
it would be interesting
to describe the precise locus where $H^1(X,\Omega^1\otimes B)$
is not zero.

We define a {\it K3 surface }to be
a smooth projective surface $X$ with
trivial canonical bundle and $H^1(X,O)=0$.
A {\it polarized K3 surface }of degree $2a$
is a K3 surface $X$ together with a primitive ample line bundle $B$ such that
$B^2=2a$.  The degree
of a polarized K3 surface must be even, because the intersection form
on $H^2(X,\Z)$ is even. Sometimes we call $(X,B)$ simply a K3 surface
of degree $2a$.

\begin{theorem}
\label{under20}
Let $X$ be a K3 surface with an ample line bundle $A$
of degree $A^2$ less than 20. Then Bott vanishing fails for $X$.
\end{theorem}

\begin{proof}
It suffices to show that $H^1(X,\Omega^1_X\otimes A)$ is not zero.
That holds if the Euler characteristic
$\chi(X,\Omega^1_X\otimes A)$ is negative. Writing $z$ for the class
of a point in $H^4(X)$, Riemann-Roch gives:
\begin{align*}
\chi(X,\Omega^1_X\otimes A)&=\int_X \td(TX)\ch(\Omega^1_X\otimes A)\\
&=\int_X (1+0+2z)(2+0-24z)(1+c_1(A)+c_1(A)^2/2)\\
&= c_1(A)^2-20.
\end{align*}
\end{proof}

We deduce the following result from the work
of Ciliberto-Dedieu-Sernesi and Feyzbakhsh \cite{CDS, Feyzbakhsh}.

\begin{theorem}
\label{lowdegree20}
Let $(X,B)$ be a polarized complex
K3 surface of degree 20 or at least 24 with Picard number 1.
Then $H^1(X,\Omega^1_X\otimes B)=0$. On the other hand,
for every polarized K3 surface $(X,B)$ of degree 22,
$H^1(X,\Omega^1_X\otimes B)\neq 0$.
\end{theorem}

Note that $B$ is a {\it primitive }ample line bundle
in Theorem \ref{lowdegree20}.
There is an irreducible
(19-dimensional) moduli space of polarized complex K3 surfaces
of degree $2a$,
for each positive integer $a$ \cite[Corollary 6.4.4]{Huybrechts}.
Moreover, a very general K3 surface $X$ in this moduli space has Picard
number 1 \cite[proof of Corollary 14.3.1]{Huybrechts}.

\begin{proof}
Let $\p_g$ be the moduli stack of pairs $(X,C)$
with $X$ a K3 surface and $C$ a smooth curve of genus $g$ in $X$
such that $O(C)$ is a primitive ample line bundle on $X$. (Then
$O(C)$ has degree $2g-2$ on $X$.) Let $\m_g$ be the moduli stack
of curves of genus $g$. There is a morphism of stacks
$$f_g\colon \p_g\arrow \m_g,$$
taking $(X,C)$ to the curve $C$. Beauville observed
that $H^1(X,\Omega^1_X\otimes O(C))^*
\cong H^1(X, TX\otimes O(-C))$
can be identified with the kernel of the derivative of $f_g$
at $(X,C)$ \cite[section 5.2]{Beauville}.
Therefore, Theorem \ref{lowdegree20} for {\it general }polarized
K3 surfaces reduces to Mukai's
theorem (completing his work with Mori)
that $f_g$ is generically finite if and only if
$g=11$ or $g\geq 13$ (corresponding to polarized K3 surfaces of degree 20
or at least 24) \cite[Theorem 1]{MM}, \cite[Theorem 7]{Mukai}.
From this point of view,
describing the locus where $H^1(X,\Omega^1_X\otimes B)$ is not zero
amounts to determining the ramification locus 
of the morphism $f_g$.

Arbarello-Bruno-Sernesi and Feyzbakhsh
recently strengthened Mukai's result by
showing that when $g=11$ or $g\geq 13$, the morphism
$f_g$ is injective at all pairs $(X,C)$ with
$X$ of Picard number 1 \cite{ABSMukai, Feyzbakhsh}. We want to show
that when $g=11$ or $g\geq 13$,
the {\it derivative }of $f_g$
is also injective at all pairs $(X,C)$ with
$X$ of Picard number 1.

The failure of Bott vanishing for K3 surfaces
$(X,B)$ of degree 22 follows from the existence of a smooth Fano 3-fold $W$
with Picard group generated by $-K_W$
and genus 12 \cite[Proposition 6]{Mukai}.
(The genus $g$ is defined
by $(-K_W)^3=2g-2$. The possible genera of smooth Fano 3-folds
with Picard group generated by $-K_W$
are $2\leq g\leq 10$ and $g=12$.)
Indeed, Beauville showed by a short deformation-theory argument
that a general hyperplane section
of a general deformation $W'$ of $W$ gives a general K3 surface $X$
of degree 22. But then a hyperplane section $C\subset X$
is the intersection of $W'$ with a codimension-2 linear space.
So there is a whole $\P^1$ of K3 surfaces (generically not isomorphic)
which all have the same curve $C$ as a hyperplane section.
That is, $f_{12}\colon \p_{12}\arrow \m_{12}$ is not generically finite,
and hence Bott vanishing fails for all K3 surfaces $(X,B)$
of degree 22.

Now let $(X,B)$ be any polarized K3 surface of degree 20 or at least 24
with Picard number 1. The assumptions imply that any smooth curve
$C$ in the linear system of $B$ has Clifford index at least 3,
as discussed in section \ref{74}. Using this together with $B^2\geq 20$,
Ciliberto-Dedieu-Sernesi showed that $h^1(X,\Omega^1\otimes B)
= \dim(\ker(df_g|_{(X,C)}))$ is equal to the fiber dimension
$\dim(f_g^{-1}(C))$ near $(X,C)$ \cite[Theorem 2.6]{CDS}.
Using that $X$ has Picard number 1 and $B^2$ is 20 or at least 24,
Arbarello-Bruno-Sernesi and Feyzbakhsh showed that $C$ lies
on a unique K3 surface of Picard number 1 \cite{ABSMukai},
\cite[Theorem 1.1]{Feyzbakhsh}. Therefore,
$f_g^{-1}(C)$ is a single point in a neighborhood of $(X,C)$.
Combining these two results
shows that $H^1(X,\Omega^1\otimes B)=0$.
\end{proof}

We now deduce the full statement
of Bott vanishing for K3 surfaces with Picard number 1:

\begin{theorem}
\label{k3.20}
Let $X$ be a complex polarized
K3 surface of degree 20 or at least 24 with Picard number 1.
Then $X$ satisfies Bott vanishing.
\end{theorem}

Note that, without the assumption of Picard number 1,
Bott vanishing does not hold for any nonempty Zariski open subset
of the moduli space of K3 surfaces of given degree $2a\geq 20$.
Indeed, there is a countably infinite set of divisors
in that moduli space corresponding to K3 surfaces
that also have an ample line bundle of degree $<20$, and Bott vanishing
fails for those K3 surfaces by Theorem \ref{under20}.

On the other hand, it is arguably more natural
to ask when Bott vanishing holds for positive multiples
of the given line bundle $B$, rather than
for all ample line bundles.
By Lemma \ref{multiple}, it is equivalent
to determine the locus of polarized K3 surfaces $(X,B)$
such that $H^1(X,\Omega^1_X\otimes B)$ is not zero.
The rest of the paper will focus on that problem.

\begin{proof} (Theorem \ref{k3.20})
Write $\Pic(X)=\Z\cdot B$ with $B$ ample.
Then every ample line bundle on $X$ is a positive multiple
of $B$.

Kodaira vanishing (Theorem \ref{kan})
gives that $H^i(X,\Omega^2_X\otimes L)=0$
for $L$ ample and $i>0$. Since $K_X=\Omega^2_X$ is trivial,
it follows that $H^i(X,L)=0$ for $L$ ample and $i>0$. Next,
KAN vanishing (Theorem \ref{kan}) gives that
$H^2(X,\Omega^1_X\otimes L)=0$ for $L$ ample.

It remains to show that $H^1(X,\Omega^1_X\otimes L)=0$
for every ample line bundle $L$ on $X$. By Theorem \ref{lowdegree20},
since $B^2$ is 20 or at least 24,
we know
that $H^1(X,\Omega^1_X\otimes B)=0$, where $B$ is the ample generator.
We will go from there to
the result for all positive multiples of $B$
(thus for all ample line bundles on $X$).

We recall Saint-Donat's sharp results about linear systems on K3 surfaces
\cite{SD}, \cite[Theorem 5]{Mori}:

\begin{theorem}
\label{sd}
Let $X$ be a K3 surface over an algebraically closed field
of characteristic not 2.
Let $B$ be a nef line bundle on $X$. Then:

(1) $B$ is not basepoint-free if and only if there is a curve
$E$ in $X$ such that $E^2=0$ and $B\cdot E=1$.

(2) Assume that $B^2\geq 4$. Then $B$ is not very ample
if and only if there is (a) a curve $E$
with $E^2=0$ such that $B\cdot E$ is 1 or 2, (b) a curve $E$
such that $E^2=2$ and $B\sim 2E$, or (c) a curve $E$
such that $E^2=-2$ and $E\cdot B=0$. (So, if $B$ is ample
and $B^2\geq 10$, $B$ fails to be very ample if and only if
there is a curve $E$ in $X$ such that $E^2=0$ and $B\cdot E$
is 1 or 2.)
\end{theorem}

The following lemma completes the proof of Theorem \ref{k3.20}.
\end{proof}

\begin{lemma}
\label{multiple}
Let $X$ be a K3 surface with a basepoint-free ample line bundle $B$.
(In particular, this holds if $\Pic(X)=\Z\cdot B$ and $B$ is ample.)
If $H^1(X,\Omega^1_X\otimes B)=0$,
then $H^1(X,\Omega^1_X\otimes B^{\otimes j})=0$
for all $j\geq 1$.
\end{lemma}

\begin{proof}
First, if $\Pic(X)=\Z\cdot B$ and $B$ is ample, then
$B$ is basepoint-free by Theorem \ref{sd}. 

Let $X$ be a K3 surface with a basepoint-free ample line bundle $B$.
By Bertini's theorem, there
is a smooth curve $D$ in the linear system $|B|$. This gives
a short exact sequence of sheaves $0\arrow O_X\arrow B\arrow B|_D\arrow 0$.
Tensoring with $\Omega^1_X$ and taking cohomology gives an exact
sequence
$$H^1(X,\Omega^1_X\otimes B)\arrow H^1(D,\Omega^1_X\otimes B)
\arrow H^2(X,\Omega^1_X).$$
We are given that $H^1(X,\Omega^1\otimes B)$ is zero,
and $H^2(X,\Omega^1_X)$ is zero since $X$ is a K3 surface; so
$H^1(D,\Omega^1_X\otimes B)=0$. Next, since $B$ restricted to the curve $D$
is basepoint-free, it is represented by an effective divisor $S$ on $D$.
This gives
a short exact sequence of sheaves $0\arrow O_D
\arrow B|_D\arrow B|_S\arrow 0$, and hence a surjection
$$H^1(D,\Omega^1_X\otimes B^{\otimes j-1})\arrow H^1(D,
\Omega^1_X\otimes B^{\otimes j})$$
for any $j\in \Z$ (using that $S$ has dimension 0). It follows
that $H^1(D,\Omega^1_X\otimes B^{\otimes j})=0$ for all $j\geq 1$.
We can then apply the analogous exact sequence on $X$,
$$H^1(X,\Omega^1_X\otimes B^{\otimes j-1})\arrow
H^1(X,\Omega^1_X\otimes B^{\otimes j})\arrow H^1(D,\Omega^1_X
\otimes B^{\otimes j}),$$
to conclude by induction that $H^1(X,\Omega^1_X\otimes B^{\otimes j})=0$
for all $j\geq 1$.
\end{proof}

\section{Failure of Bott vanishing on a K3 surface
in terms of elliptic curves of low degree}
\label{74}

Theorem \ref{fano} clarifies the meaning of Bott vanishing
for a K3 surface $X$. Namely,
if $H^1(X,\Omega^1_X\otimes B)\neq 0$
for an ample line bundle $B$, then one of three conditions
must hold: $B^2$ is less than 20,
there is an elliptic curve of low degree with respect to $B$,
or $X$ is an anticanonical divisor in a singular Fano 3-fold $Y$
with $B=-K_Y|_X$.
The proof is based on recent work of Ciliberto, Dedieu,
and Sernesi, which in turn buids on the work
of Arbarello, Bruno, and Sernesi \cite{ABShyperplane, CDS}.

The main result of Arbarello-Bruno-Sernesi was that a Brill-Noether general
curve $C$ of genus at least 12 is the hyperplane section of a (possibly
singular) K3 surface,
or of a limit of K3 surfaces (a ``fake K3''), if and only if the Wahl
map of $C$ is not surjective. The proof was based on a new vanishing
theorem for the square of the ideal sheaf of a projective curve.
Ciliberto-Dedieu-Sernesi applied that work on curves
to give criteria
for a projective K3 surface to be a hyperplane section of a
(possibly singular) Fano 3-fold.

Theorem \ref{fano} characterizes exactly when $H^1(X,\Omega^1\otimes B)$
is zero except when $X$ contains an elliptic curve of low degree. That case
is studied in section \ref{ellipticsection}, which includes a complete answer
for $B$ of high degree.

The classification of Fano 3-folds with canonical Gorenstein
singularities remains open. As a result, Theorem \ref{fano}
is not as explicit an answer as one might like. Nonetheless,
it is a strong statement, from which we draw more specific
consequences in the rest of the paper. For our purpose,
we only want the classification of Fano 3-folds with
{\it isolated }canonical Gorenstein singularities, which
may be within reach.

In particular, Theorem \ref{fano} implies that for K3 surfaces $(X,B)$
with no elliptic curve of low degree, the nonvanishing
of $H^1(X,\Omega^1\otimes B)$ is a Noether-Lefschetz
condition. More precisely, this group is nonzero if and only
if $(X,B)$ belongs to certain irreducible components
of the space of K3 surfaces with Picard group containing
one of a finite list of lattices. (This follows from Theorem \ref{fano}
by Beauville's deformation-theory
argument, which works with no change for Fano 3-folds with
isolated singularities \cite[Theorem]{Beauville}.)
The lattices that occur are exactly the Picard
lattices of the Fano 3-folds with isolated canonical Gorenstein
singularities, these being not yet known.

\begin{theorem}
\label{fano}
Let $X$ be a complex K3 surface with an ample line bundle $B$ such that
there is no curve $E$ in $X$ with $E^2=0$
and $1\leq B\cdot E \leq 4$.
Then $H^1(X,\Omega^1\otimes B)\neq 0$ if and only if $B^2<20$
or $X$ is a smooth anticanonical divisor in some Fano 3-fold $Y$
with at most isolated canonical Gorenstein singularities
such that $B=-K_Y|_X$.
\end{theorem}

Note that any curve $E$ in a K3 surface
with $E^2=0$ is a fiber of an elliptic fibration,
for example by Theorem \ref{sd}. Using work of Prokhorov,
the final case of Theorem \ref{fano} implies that $B^2\leq 72$
(Theorem \ref{lowdegree74}).

\begin{proof}
If $B^2<20$, then $H^1(X,\Omega^1_X\otimes B)\neq 0$
by Riemann-Roch (Theorem \ref{under20}). If $X$ is a smooth
anticanonical divisor in some Fano 3-fold $Y$
with at most isolated canonical Gorenstein singularities
such that $B=-K_Y|_X$, then $H^1(X,\Omega^1_X\otimes B)\neq 0$.
This follows from Lvovski's theorem on extensions of projective
varieties \cite[Theorem 0.1, Lemma 3.5]{CDS}. (One could also
prove this by extending Mukai's argument from section \ref{bott20}.)
Conversely, assume that $B^2\geq 20$. We want to show
that if $H^1(X,\Omega^1_X\otimes B)\neq 0$, then $X$
is an anticanonical divisor.

By Theorem \ref{sd}, $B$ is very ample, giving an embedding
$X\subset \P^g$ where $B^2=2g-2$. Choose a smooth hyperplane
section $C$ in $X$ (so $C$ has genus $g$, and the embedding
$C\arrow \P^{g-1}$ is the canonical embedding).

For a line bundle $L$ on a smooth projective curve $C$,
the {\it Clifford index }$\Cliff(C,L)$ is $\deg(L)-2h^0(C,L)+2$.
For $C$ of genus at least 4, the {\it Clifford index }of $C$ is
$$\Cliff(C):=\min \{ \Cliff(C,L):h^0(C,L)\geq 2
\text{ and }h^1(C,L)\geq 2\}.$$

I claim that the curve $C\subset X$ above
has Clifford index at least 3. Several approaches are possible,
but we use the following result of Knutsen, inspired by earlier
work of Green-Lazarsfeld and Martens \cite[Lemma 8.3]{Knutsen},
\cite{GL, Martens}.

\begin{lemma}
\label{knutsen}
Let $B$ be a basepoint-free line bundle on a K3 surface $X$ with
$B^2=2g-2\geq 2$. Let $c$ be the Clifford index of a smooth curve $C$
in $|B|$.

If $c<\lfloor (g-1)/2\rfloor$, then there is a smooth curve $E$ on $X$
such that $0\leq E^2\leq c+2$ and $B\cdot E=E^2+c+2$.
\end{lemma}

Since we have $B^2\geq 20$, Lemma \ref{knutsen} together with the Hodge index
theorem implies that if $C$ has Clifford index at most 2,
then the curve $E$ given by the lemma
has $E^2=0$ and hence $1\leq B\cdot E\leq 4$,
contradicting our assumptions. So $C$ has Clifford index at least 3.

For a smooth projective curve $C$, the {\it Wahl map}
$$\Phi_C\colon \Lambda^2H^0(C,K_C)\arrow H^0(C,K_C^{\otimes 3})$$
is defined by $s\wedge t\mapsto s\, dt-t\, ds$. Wahl showed
that the Wahl map of a curve contained in some K3 surface
is not surjective; that is, $\corank(\Phi_C)\geq 1$ \cite{Wahl}.
When $g\geq 11$
and $\Cliff(C)\geq 3$ (as here), Ciliberto-Dedieu-Sernesi
proved the more precise statement:
$$\corank(\Phi_C)=h^1(X,\Omega^1_X\otimes B)+1$$
\cite[Corollary 2.8]{CDS}.

Let $r=h^1(X,\Omega^1\otimes B)$. By Ciliberto-Dedieu-Sernesi,
again using that $g\geq 11$ and $\Cliff(C)\geq 3$,
there is an arithmetically Gorenstein normal variety $Z$
of dimension $r+2$ in $\P^{g+r}$, not a cone,
containing the curve $C\subset \P^{g-1}$ as the section
by a linear space of dimension $g-1$. Moreover, $Z$ contains
$X\subset \P^g$ as the section by a linear space of dimension
$g$ \cite[section 2.2]{CDS}.

Thus, if $H^1(X,\Omega^1\otimes B)\neq 0$, then $Z$
has dimension at least 3. Let $Y$ be the intersection of $Z$
with a general linear space of dimension $g+1$ that contains
$X$; then $Y\subset \P^{g+1}$ is an arithmetically Gorenstein 3-fold
with $-K_Y=O(1)$. Because
$Y$ has a smooth hyperplane section $X$ and $Z$ is not a cone,
$Y$ has at most
isolated canonical singularities \cite[Corollary 5.6]{CDS}.
Theorem \ref{fano} is proved.
\end{proof}

\section{K3 surfaces of high degree}

In the rest of the paper, we analyze which ample line bundles $B$
on a K3 surface $X$ have $H^1(X,\Omega^1\otimes B)=0$,
without assuming that $X$ has Picard number 1. We give complete answers
when $B$ has high enough degree. This section proves a first step,
as an immediate consequence of Prokhorov's work on Fano 3-folds:
in high degrees, the locus where $H^1(X,\Omega^1\otimes B)\neq 0$
is contained in the locus of K3 surfaces
that contain ``low-degree'' elliptic curves.
This is analogous to Saint-Donat's theorem on very ampleness,
Theorem \ref{sd}. Theorem \ref{degree4} will analyze
the case when there is a low-degree elliptic curve.

\begin{theorem}
\label{lowdegree74}
Let $B$ be an ample line bundle on a complex K3 surface $X$
with $B^2\geq 74$. If $H^1(X,\Omega^1_X\otimes B)\neq 0$,
then there is a curve $E$ in $X$ with $E^2=0$ and $1\leq B\cdot E\leq 4$.
\end{theorem}

As mentioned earlier, any such curve $E$ is a fiber of an elliptic fibration
of $X$.

\begin{proof}
Suppose that there is no curve $E$ in $X$ with $E^2=0$
and $1\leq B\cdot E\leq 4$, and that $H^1(X,\Omega^1_X\otimes B)$
is not zero. By Theorem \ref{fano}, $X$ is a smooth anticanonical
divisor in some Fano 3-fold $Y$
with at most isolated canonical Gorenstein singularities
such that $B=-K_Y|_X$.

Prokhorov showed that a Fano 3-fold $Y$
with canonical Gorenstein singularities has $(-K_Y)^3\leq 72$
\cite[Theorem 1.5]{Prokhorov}. (For comparison, a smooth
Fano 3-fold $Y$ has $(-K_Y)^3\leq 64$.)
So we reach a contradiction if $B^2\geq 74$.
\end{proof}

The degree bound 74 in Theorem \ref{lowdegree74} is sharp,
by the following example.

\begin{example}
Let $X$ be the double cover of $\P^2$ ramified over a very general
sextic curve. Let $A$ be the pullback of the line bundle
$O_{\P^2}(1)$. Here $(X,A)$ is a polarized K3 surface of degree 2
and Picard number 1. Then
the line bundle $B=6A$ has $B^2=72$ and $H^1(X,
\Omega^1_X\otimes B)\neq 0$, while there is no curve $E$
in $X$ with $E^2=0$ and $1\leq B\cdot E\leq 4$.
(Thus Theorem \ref{lowdegree74} fails in degree 72.)

Proof: By the assumption of generality, $\Pic(X)=\Z\cdot A$.
So there is no curve $E$ in $X$ with $E^2=0$.

By considering the graded ring associated to $A$, $X$
embeds as a hypersurface of degree 6 in the weighted projective space
$Y=P(3,1,1,1)$.
Here $Y$ is a Fano 3-fold with canonical Gorenstein singularities,
$-K_Y=O(6)$,
and $(-K_Y)^3=72$ \cite[Theorem 1.5]{Prokhorov}. In particular,
$B=-K_Y|_X$.

By Lvovski's theorem as in the proof of Theorem \ref{fano},
because $(X,B)$
is an anticanonical section of a Fano 3-fold $Y$, we have
$H^1(X,\Omega^1_X\otimes B)\neq 0$. (Ciliberto-Lopez-Miranda claimed
that $H^1(X,\Omega^1_X\otimes B)=0$ in this case, because of an error
in the proof of \cite[Lemma 2.3(e)]{CLM}: in the description
of the tangent bundle of a ramified cover, $N_{\pi}$ should
be $\pi^*O_B(6)$, not $\pi^*O_B(3)$.)

Thus we have examples of K3 surfaces $X$
with Picard number 1 and an ample line bundle $B$ of degree 72
such that $H^1(X,\Omega^1_X\otimes B)\neq 0$, showing
the optimality of Theorem \ref{lowdegree74}. This does not contradict
Theorem \ref{lowdegree20}, because $B=6A$ is not primitive.
\end{example}

\begin{example}
There is a K3 surface $X$ with a {\it primitive }ample
line bundle $B$ of degree 62 such that $H^1(X,\Omega^1_X\otimes B)\neq 0$
and there is no curve $E$ in $X$ with $E^2=0$
and $1\leq B\cdot E\leq 4$.
(Thus Theorem \ref{lowdegree74} fails for primitive ample
line bundles of degree 62.)

Proof: Let $Y$
be the Fano 3-fold $P(O\oplus O(2))\arrow \P^2$, which has
$(-K_Y)^3=62$ and $-K_Y$ primitive.
Let $X$ be a smooth divisor in the linear system
of $-K_Y$. Then $X$ is a K3 surface with a primitive ample
line bundle $B=-K_Y|_X$ of degree 62, and $H^1(X,\Omega^1\otimes B)\neq 0$
by Lvovski's theorem again. For $X$ very general, the restriction
homomorphism $\Pic(Y)=\Z\{R,S\}\arrow \Pic(X)$ is an isomorphism. Given that,
it is straightforward to compute the intersection form on $X$ (it has 
$R^2=2$, $RS=5$, and $S^2=10$). This quadratic form
does not represent zero nontrivially,
and so $X$ contains no curve $E$ with $E^2=0$.
Thus Theorem \ref{lowdegree74}
fails for $(X,B)$, as promised.
\end{example}

\section{Elliptic K3 surfaces}
\label{ellipticsection}

We now analyze which K3 surfaces $(X,B)$ have
$H^1(X,\Omega^1_X\otimes B)\neq 0$
when there is a curve $E$ in $X$ with
$E^2=0$ and $1\leq B\cdot E\leq 4$; these are the cases
left out of Theorem \ref{lowdegree74}. The answer is complete
if $B\cdot E=1$ or also if $B^2$ is large enough (with
explicit bounds). Surprisingly, the answer depends on
whether an elliptic fibration of $X$ has a certain
special type of singular fiber.

In particular, when $1\leq B\cdot E\leq 3$,
we give examples with $B^2$ arbitrarily large such that
$H^1(X,\Omega^1_X\otimes B)\neq 0$, showing that these cases are genuine
exceptions to Theorem \ref{lowdegree74}. By contrast,
when $B\cdot E=4$, this cohomology group is in fact
zero for $B^2\geq 194$. (This bound is probably not optimal.)

\begin{theorem}
\label{degree4}
Let $B$ be an ample line bundle on a complex K3 surface $X$.
Suppose that there is a curve $E$ in $X$ with $E^2=0$
and $1\leq B\cdot E\leq 4$.
Let $\pi\colon X\arrow \P^1$ be the elliptic fibration
associated to $E$. If $r=1$ and $\pi$ has a fiber of type II,
or $r=2$ and $\pi$ has a fiber of type III, or $r=3$ and $\pi$
has a fiber of type IV, then $H^1(X,\Omega^1_X\otimes B)\neq 0$.
The converse holds if in addition $r=1$ and $B^2\geq 40$,
or $r=2$ and $B^2\geq 92$, or $r=3$ and $B^2\geq 140$,
or $r=4$ and $B^2\geq 194$.
\end{theorem}

In Kodaira's classification of the singular fibers
of an elliptic surface, type II is a cuspidal cubic curve,
type III is two copies of $\P^1$ tangent at a point,
and type IV is three copies of $\P^1$ through a point
\cite[Corollary 5.2.3]{CD}.

We first consider the case where $B\cdot E=1$, in which case
$(X,B)$ is said to be {\it unigonal}. In this case,
we have an even stronger statement than Theorem \ref{degree4}:
we can describe exactly when $H^1(X,\Omega^1_X\otimes B)$ is not zero,
without having to assume that $B^2\geq 40$.
Most of the proof of the following theorem
was suggested by Ben Bakker.

\begin{theorem}
\label{degree1}
Let $B$ be an ample line bundle on a complex K3 surface $X$.
Suppose that there is a curve $E$ in $X$ with $E^2=0$
and $B\cdot E=1$. Let $\pi\colon X\arrow \P^1$ be the elliptic fibration
associated to $E$. Then $H^1(X,\Omega^1_X\otimes B)\neq 0$
if and only if $B^2\leq 38$ or some fiber of $\pi$
is of type II (a cuspidal cubic).
\end{theorem}

In particular, there are polarized K3 surfaces $(X,B)$ with $B^2$ arbitrarily
large such that there is a curve $E$ in $X$ with $E^2=0$
and $B\cdot E=1$ while $H^1(X,\Omega^1\otimes B)\neq 0$.
To construct such examples, let $\pi\colon X\arrow \P^1$
be an elliptic K3 surface with a section $B_0$ such that
there are 22 fibers of type I$_1$ (a nodal cubic)
and one fiber of type II (a cuspidal cubic).
Such a surface is easy to construct, using a Weierstrass equation.
Let $E$ be a fiber of $\pi$; then $B_0^2=-2$, $E^2=0$, and $B_0\cdot E=1$.
For any integer $m\geq 2$, it is straightforward to check
that $B:=B_0+mE$ is ample, and we have
$B^2=2m-2$ and $B\cdot E=1$. Since $B\cdot E=1$, $B$ is primitive.
By Theorem \ref{degree1},
$H^1(X,\Omega^1_X\otimes B)$ is not zero, no matter how big $B^2$ is.

Theorem \ref{degree1} shows that the locus of polarized K3
surfaces $(X,B)$ with $H^1(X,\Omega^1_X\otimes B)\neq 0$
is not a Noether-Lefschetz locus when there is an elliptic
curve of low degree. (That is, this property
cannot always be read from the Picard lattice of $X$.)
Indeed, the condition that an elliptic fibration $\pi\colon X\arrow \P^1$
has a cuspidal fiber is not determined by the Picard lattice
of $X$. A general elliptic K3 surface as in the previous paragraph has
Picard lattice $\Z\cdot\{B_0,E\}$ with $B_0^2=-2$, $E^2=0$,
and $B_0\cdot E=1$,
whether there is a cuspidal fiber or not.

\begin{proof}
(Theorem \ref{degree1})
By the Riemann-Roch calculation in Theorem \ref{under20},
we know that $H^1(X,\Omega^1_X\otimes B)\neq 0$ if $B^2<20$.
So we can assume from now on that $B^2\geq 20$.

Every fiber of $\pi\colon X\arrow \P^1$ is an effective divisor
linearly equivalent to $E$. Since $B$ is ample and $B\cdot E=1$,
every fiber of $\pi$ is irreducible and has multiplicity 1.
By Kodaira's classification, every singular fiber of $\pi$ is of type
I$_1$ (a nodal cubic) or II (a cuspidal cubic).

By Riemann-Roch, $B$ can be represented by an effective divisor.
Since $B\cdot E=1$, this divisor must be the sum of a section
$B_0$ of $\pi$ with some curves supported in fibers.
Then $B_0\cong \P^1$, and so $B_0^2=-2$.
Because all fibers of $\pi$
are irreducible, it follows that $B$ is linearly equivalent to
$B_0+mE$, where $B^2=2m-2$.

We have the following exact sequence of coherent sheaves on $X$:
$$0\arrow \pi^*\Omega^1_{\P^1}\arrow \Omega^1_X\arrow \Omega^1_{X/\P^1}
\arrow 0.$$
The sheaf $\Omega^1_{X/\P^1}$ of relative K\"ahler differentials
is torsion-free but not reflexive, by a direct computation
at singular fibers of $\pi$. It is
related to the relative dualizing sheaf $\omega_{X/\P^1}$ (a line bundle)
by another exact sequence:
$$0\arrow \Omega^1_{X/\P^1}\arrow \omega_{X/\P^1}\arrow \omega_{X/\P^1}|_S
\arrow 0,$$
where $S$ is the non-smooth locus of $\pi$. Here $S$ is
a closed subscheme of degree 24 in $X$, supported at the singular
points of fibers of $\pi$. (One can compute directly that $S$ has degree 2
at cusps and degree 1 at nodes.)

Since the line bundle $\Omega^1_{\P^1}$ is isomorphic to $O(-2)$,
$\pi^*\Omega^1_{\P^1}$ is isomorphic to $O(-2E)$. Tensoring
the first exact sequence with $B$ and taking cohomology gives
an exact sequence of complex vector spaces:
$$H^1(X,B-2E)\arrow H^1(X,\Omega^1_X\otimes B)\arrow
H^1(X,\Omega^1_{X/\P^1}\otimes B)
\arrow H^2(X,B-2E).$$
We arranged that $B^2\geq 20$, and so $m\geq 11$. (For what follows,
$m\geq 4$ would be enough.) Therefore, $B-2E
=B_0+(m-2)E$ is nef and big. So $H^1(X,B-2E)=H^2(X,B-2E)=0$
by Kawamata-Viehweg vanishing. We deduce that $H^1(X,\Omega^1_X\otimes B)$
maps isomorphically to $H^1(X,\Omega^1_{X/\P^1}\otimes B)$.

Outside the 0-dimensional subscheme $S$ of $X$, the first exact sequence
above is an exact sequence of vector bundles. Taking determinants
shows that $\Omega^1_{X/\P^1}$ is isomorphic to $O(2E)$ outside $S$,
using that $K_X$ is trivial. Because $\omega_{X/\P^1}$ is a line bundle
on all of $X$, it follows that $\omega_{X/\P^1}\cong O(2E)$. So the second
exact sequence (tensored with $B$) gives a long exact sequence of cohomology:
$$H^0(X,O(B+2E))\arrow H^0(S,O(B+2E))\arrow H^1(X,\Omega^1_X\otimes B)
\arrow H^1(X,O(B+2E)).$$
Here $B+2E$ is nef and big, and so the last cohomology group is zero.
We conclude that $H^1(X,\Omega^1_X\otimes B)=0$ {\it if and only if
the subscheme $S$ imposes linearly independent conditions
on sections of the line bundle $B+2E$.} Thus for
elliptic K3s, Bott vanishing reduces to a question about
sections of a line bundle, which is much easier to analyze.

In the case at hand, we can describe all sections of $B+2E=B_0+(m+2)E$
explicitly. We have $h^0(L)=(L^2+4)/2$ for $L$ nef and big
on a K3 surface $X$, and so $h^0(B+2E)=m+3$.
But we get an $(m+3)$-dimensional
space of sections of $O(B+2E)=O(B_0)\otimes \pi^*O(m+2)$
by pulling back sections of $O(m+2)$ on $\P^1$, and so those
are all the sections. In other words, the linear
system of $B+2E$ is exactly the set of divisors
$B_0+E_1+\cdots+E_{m+2}$ for some fibers
$E_1,\ldots, E_{m+2}$ of $\pi$.

If $\pi$ has a fiber $E_0$ with a cusp $p$, then the subscheme $S$
has degree 2 at $p$ (and is contained in $E_0$); so $S$ does
not impose linearly independent conditions on sections of $B+2E$
in this case. Otherwise, all singular fibers of $\pi$ have
a single node, and so $S$ consists of 24 points in distinct
fibers of $\pi$. It follows that $S$ imposes linearly independent
conditions on sections of $B+2E$ if and only if $m+2\geq 23$,
that is, $B^2\geq 40$.
\end{proof}

We now address the cases where $B\cdot E$ is 2, 3, or 4.
The K3 surface $(X,B)$ is said to be {\it hyperelliptic},
{\it trigonal}, or {\it tetragonal}, respectively (because
all smooth curves in the linear system of $B$ have the given
gonality).

Before proving Theorem \ref{degree4}, we use it to give examples
such that $B\cdot E$ is 1, 2, or 3
and $H^1(X,\Omega^1_X\otimes B)\neq 0$
for arbitrarily
large values of $B^2$, in contrast to Theorem \ref{lowdegree74}.
(This was done above when $B\cdot E=1$.)
When $B\cdot E=4$, by contrast, the theorem says that
$H^1(X,\Omega^1_X\otimes B)=0$ whenever $B^2\geq 194$.

\begin{example}
There are polarized K3 surfaces $(X,B)$ with $B^2$ arbitrarily
large such that there is a curve $E$ in $X$ with $E^2=0$ and $B\cdot E=2$,
while $H^1(X,\Omega^1_X\otimes B)\neq 0$.

Let $X$ be the double cover of $Y=\P^1\times \P^1$
ramified along a smooth curve $D$ in the linear system
of $-2K_Y=O(4,4)$. Then $X$ is a K3 surface, with two elliptic fibrations
defined by the two compositions $X\arrow Y\arrow \P^1$.
Write $\pi\colon X\arrow \P^1$ for the first fibration,
$E$ for a fiber of $\pi$, and $C_0$ for a fiber of the second
fibration; then $C_0\cdot E=2$. Let $S$ be the non-smooth locus
of $\pi$, a closed subscheme
of degree 24 in $X$.
By choosing $D$ to have intersection with one curve
$p\times \P^1$ equal to a single point with multiplicity 4,
we can arrange that the corresponding fiber of $\pi$
is of type III (two $\P^1$s tangent at one point).
By Theorem \ref{degree4}, $H^1(X,\Omega^1_X\otimes B)\neq 0$,
while $B^2=4m$ can be arbitrarily large. (One can give a similar
example with $B^2\equiv 2\Mod{4}$ by taking $X$ to be a double cover
of $P(O\oplus O(1))\arrow \P^1$, rather than of $\P^1\times \P^1$.)
\end{example}

\begin{example}
There are polarized K3 surfaces $(X,B)$ with $B^2$ arbitrarily
large such that there is a curve $E$ in $X$ with $E^2=0$ and $B\cdot E=3$,
while $H^1(X,\Omega^1_X\otimes B)\neq 0$.

To see this, let $X$ be a smooth anticanonical divisor in $\P^1\times \P^2$
such that one fiber $E_0$ of the elliptic fibration $\pi\colon X\arrow \P^1$
consists of three lines through a point (thus, a fiber of type IV).
By Theorem \ref{degree4}, $H^1(X,\Omega^1_X\otimes B)\neq 0$,
while $B^2=6m+2$ can be arbitrarily large.
\end{example}

\begin{proof} (Theorem \ref{degree4})
Let $r=B\cdot E$. For $r=1$, the theorem follows from Theorem
\ref{degree1}. From now on, assume that $2\leq r\leq 4$.

We use the following analysis of K3 surfaces of low Clifford genus,
due to Reid, Brawner, and Stevens \cite[section 2.11]{ReidChapters},
\cite[Tables A.1-A.4]{Brawner}, \cite[table in section 1]{Stevens}.

\begin{theorem}
\label{nef}
Let $X$ be a complex K3 surface with a line bundle $L$. Suppose that
there is a curve $E$ in $X$ with $E^2=0$
and $r:=B\cdot E$ between 1 and 4. Suppose that $L+sE$
is ample for some integer $s$. Finally, suppose that $r=1$ and $L^2\geq 2$,
or $r=2$ and $L^2\geq 8$,
or $r=3$ and $L^2\geq 14$,
or $r=4$ and $L^2\geq 26$. Then $L$ is nef,
and $h^0(L)-h^0(L-E)=r$.
\end{theorem}

\begin{proof}
The references cited determine the possible
values of the sequence of integers $h^0(L+mE)$. (For $r\geq 2$,
that sequence describes the scroll
$P(O(e_1)\oplus \cdots O(e_r))\arrow \P^1$
that contains the image of $X$ under the morphism to projective space
given by $L+mE$ for $m$ large.) In particular, these results say that
$h^0(L)-h^0(L-E)=r$ under our assumption on $L^2$. It follows that
$h^0(L-E)$ is given by Riemann-Roch and hence that
$h^1(L-E)=0$ (because $h^2(L-E)$ is easily seen to be zero).
By Knutsen and Lopez's characterization of line bundles with
vanishing cohomology on a K3 surface,
it follows that $L-E$ has degree at least $-1$
on any $(-2)$-curve in $X$ \cite[Theorem]{KL}. Using that plus the fact
that $L+mE$ is ample for $m$ large, we deduce that $L$ is nef.
\end{proof}

As in the proof of Theorem \ref{degree1}, let $S$ be the non-smooth
locus of $\pi$, viewed as a closed subscheme of degree 24 in $X$,
supported at the singular points of fibers of $\pi$.
We compute that $S$ has degree 1 at nodes, 2 at cusps (on fibers
of type II),
3 at type III, and 4 at type IV. Moreover, each connected component
of $S$ is contained (as a scheme) in a fiber of $\pi$.

By the proof of Theorem \ref{degree1}, if $H^1(X,\Omega^1_X\otimes B)$
is zero, then $S$ imposes linearly independent conditions
on sections of $B+2E$. Moreover, the converse holds if $B-2E$ is nef
and big. Suppose that $r=2$ and $\pi$ has a fiber of type III, or $r=3$
and $\pi$ has a fiber of type IV. 
(We are assuming $r\geq 2$ now, but the argument would be the same
in the case where
$r=1$ and $\pi$ has a fiber of type II.)
Let $S_0$ be the connected component
of $S$ at the given singular point. Then $S_0$ has degree $r+1$.
On the other hand, the line bundle $B+2E$ is ample and has degree $r$
on the given fiber $E_0$ (which has $r$ irreducible components), and so it
has degree only 1 on each component. It follows that
$h^0(E_0,B+2E)$ is only $r$. So the restriction map
$H^0(X,B+2E)\arrow H^0(S_0,B+2E)=\C^{r+1}$ is not surjective.
So $S$ does not impose independent conditions on sections
of $B+2E$, and hence $H^1(X,\Omega^1_X\otimes B)$ is not zero.
The first part of the theorem is proved.

For the converse, 
suppose that $r=2$ and $B^2\geq 92$, or $r=3$ and $B^2\geq 140$,
or $r=4$ and $B^2\geq 194$. Also, if $r=2$, assume that $\pi$ has no fiber
of type III, and if $r=3$, assume that $\pi$ has no fiber
of type IV. We want to deduce that $H^1(X,\Omega^1_X\otimes B)=0$.

Let $L=B-21E$, so that
$L^2=B^2-42r$. Thus either $r=2$ and $L^2\geq 8$,
or $r=3$ and $L^2\geq 14$, or $r=4$ and $L^2\geq 26$.
By Theorem \ref{nef}, $L$ is nef, and $h^0(L)-h^0(L-E)=r$.
So, for each fiber $E_0$ of $\pi$,
the image of the restriction $H^0(X,L)\arrow H^0(E_0,L)$
has dimension $r$.
Also, we are given that
$L+sE$ is ample, and so $L$ is ample on $E_0$, with degree $r\leq 4$.
It follows that $E_0$ has at most $r$ connected
components. So $E_0$ has type $I_n$ for $n\leq r$ or II or III
(with $r$ equal to 3 or 4) or IV (with $r=4$).
By Riemann-Roch
for 1-dimensional schemes \cite[Tag 0BS6]{Stacks} plus Serre duality,
$H^0(E_0,L)$ has dimension $r$. (Use that
$E_0$ is Gorenstein, with trivial canonical bundle.)
So $H^0(X,L)\arrow H^0(E_0,L)$ is surjective, for each fiber
$E_0$ of $\pi$.

Since $L=B-21E$ is nef and big, so is $B-2E$. Since $B-2E$ is nef and big,
the proof of Theorem \ref{degree1} shows that $H^1(X,\Omega^1_X\otimes B)=0$
(as we want) if and only if $S$ imposes independent conditions on sections
of $B+2E$. Here $B+2E=L+23E$.

Let $E_0$ be any singular fiber of $\pi$, and let $S_0=S\cap E_0$,
which is an open subscheme of $S$. We showed above
that $H^0(X,L)\arrow H^0(E_0,L)$ is surjective. Also, $L$ is ample
on $E_0$. It follows that
$H^0(E_0,L)\arrow H^0(S_0,L)$ is surjective, by inspection
of the possible types of singular fibers (since we have excluded
the case where $r=2$ and $E_0$ is of type III, or $r=3$
and $E_0$ is of type IV).
Therefore, $H^0(X,L)\arrow H^0(S_0,L)$ is surjective.
It is then clear that $H^0(X,L+23E)\arrow H^0(S,L+23E)$ is surjective,
using sections of $O(23E)$ that vanish on all singular fibers of $\pi$
except one. (We are using that the number of singular fibers is at most 24.)
Since $B+2E=L+23E$, this completes the proof that
$H^1(X,\Omega^1_X\otimes B)=0$.
\end{proof}


\small \sc UCLA Mathematics Department, Box 951555,
Los Angeles, CA 90095-1555

totaro@math.ucla.edu

\begin{thebibliography}{99}
\bibitem{AWZ} P.~Achinger, J.~Witaszek,
and M.~Zdanowicz. Liftability of the Frobenius morphism
and images of toric varieties. \url{arXiv:1708.03777}

\bibitem{ABSMukai} E.~Arbarello, A.~Bruno, and E.~Sernesi.
Mukai's program for curves on a K3 surface.
{\it Algebraic Geometry }{\bf 1 }(2014), 532--557.

\bibitem{ABShyperplane} E.~Arbarello, A.~Bruno, and E.~Sernesi.
On hyperplane sections of K3 surfaces.
{\it Algebraic Geometry }{\bf 4 }(2017), 562--596.

\bibitem{BC} V.~Batyrev and D.~Cox. On the Hodge structure
of projective hypersurfaces in toric varieties.
{\it Duke Math.\ J.\ }{\bf 75 }(1994), 293--338.

\bibitem{Beauville} A.~Beauville. Fano threefolds
and K3 surfaces. {\it The Fano conference}, 175--184.
University of Torino (2004).

\bibitem{Brawner} J.~Brawner. Tetragonal curves, scrolls
and K3 surfaces. {\it Trans.\ Amer.\ Math.\ Soc.\ }{\bf 349 }(1997),
3075--3091.

\bibitem{BTLM} A.~Buch, J.~Thomsen, N.~Lauritzen,
and V.~Mehta. The Frobenius morphism on a toric variety.
{\it Tohoku Math.\ J.\ }{\bf 49 }(1997), 355--366.

\bibitem{CDGK} C.~Ciliberto, T.~Dedieu, C.~Galati,
and A.~Knutsen. Moduli of curves on Enriques surfaces.
\url{arXiv:1902.07142}

\bibitem{CDS} C.~Ciliberto, T.~Dedieu, and E.~Sernesi.
Wahl maps and extensions of canonical curves
and K3 surfaces. {\it J.\ reine angew.\ Math.}, to appear.

\bibitem{CLM} C.~Ciliberto, A.~Lopez, and R.~Miranda.
Classification of varieties with canonical curve section
via Gaussian maps on canonical curves.
{\it Amer.\ J.\ Math. }{\bf 120 }(1998), 1--21.

\bibitem{CD} F.~Cossec and I.~Dolgachev.
{\it Enriques surfaces. }Birkh\"auser (1989).

\bibitem{Debarre} O.~Debarre. {\it Higher-dimensional
algebraic geometry. }Springer (2001).

\bibitem{DI} P.~Deligne and L.~Illusie.
Rel\`evements modulo $p^2$ et d\'ecomposition
du complexe de de Rham. {\it Invent.\ Math.\ }{\bf 89 }(1987),
247--270.

\bibitem{Feyzbakhsh} S.~Feyzbakhsh.
Mukai's program (reconstructing a K3 surface
from a curve) via wall-crossing.
\url{arXiv:1710.06692}

\bibitem{Fujino} O.~Fujino. Multiplication maps
and vanishing theorems for toric varieties.
{\it Math.\ Z.\ }{\bf 257 }(2007), 631--641.

\bibitem{Fulton} W.~Fulton.
{\it Introduction to toric varieties. }Princeton (1993).

\bibitem{GL} M.~Green and R.~Lazarsfeld. Special divisors
on curves on a K3 surface.
{\it Invent.\ Math.\ }{\bf 89 }(1987), 357--370.

\bibitem{Hartshorne} R.~Hartshorne. {\it Algebraic
geometry. }Springer (1977).

\bibitem{Huybrechts} D.~Huybrechts. {\it Lectures
on K3 surfaces. }Cambridge (2016).

\bibitem{Knutsen} A.~Knutsen. On $k$th order embeddings
of K3 surfaces and Enriques surfaces.
{\it Manu.\ Math.\ }{\bf 104 }(2001), 211--237.

\bibitem{KL} A.~Knutsen and A.~Lopez. A sharp
vanishing theorem for line bundles on K3 or Enriques
surfaces. {\it Proc.\ Amer.\ Math.\ Soc.\ }{\bf 135 }(2007),
3495--3498.

\bibitem{Lazarsfeld} R.~Lazarsfeld. {\it Positivity
in algebraic geometry}, 2 vols. Springer (2004).

\bibitem{Manin} Y.~Manin. {\it Cubic
forms, }2nd ed. North-Holland (1986).

\bibitem{Martens} G.~Martens. On curves on K3 surfaces.
{\it Algebraic curves and projective
geometry }(Trento, 1988), 174--182.
Lecture Notes in Math.\ 1387, Springer (1989).

\bibitem{Mori} S.~Mori. On degrees and genera of curves
on smooth quartic surfaces in $\P^3$.
{\it Nagoya Math.\ J.\ }{\bf 96 }(1984), 127--132.

\bibitem{MM} S.~Mori and S.~Mukai. The uniruledness
of the moduli space of curves of genus 11.
{\it Algebraic geometry }(Tokyo/Kyoto, 1982),
334--353. Lecture Notes in Math.\ 1016, Springer (1983).

\bibitem{Mukai} S.~Mukai. Fano 3-folds. {\it Complex
projective geometry }(Trieste--Bergen, 1989), 255--263.
London Math.\ Soc.\ Lecture Note Ser.\ 179, Cambridge (1992).

\bibitem{Mustata} M.~Musta\c{t}\u{a}. Vanishing theorems
on toric varieties. {\it Tohoku Math.\ J.\ }{\bf 54 }(2002),
451--470.

\bibitem{Prokhorov} Y.~Prokhorov. The degree
of Fano threefolds with canonical Gorenstein
singularities. {\it Mat.\ Sb.\ }{\bf 196 }(2005), 81--122.

\bibitem{ReidChapters} M.~Reid. Chapters on algebraic surfaces.
{\it Complex algebraic geometry }(Park City, 1993), 3--159.
Amer.\ Math.\ Soc.\ (1997).

\bibitem{SD} B.~Saint-Donat. Projective models of K-3 surfaces.
{\it Amer.\ J.\ Math.\ }{\bf 96 }(1974), 602--639.

\bibitem{Stacks} The Stacks Project Authors. {\itshape Stacks
Project }(2018). \url{http://stacks.math.columbia.edu}

\bibitem{Stevens} J.~Stevens. Rolling factors deformations
and extensions of canonical curves.
{\it Doc.\ Math.\ }{\bf 7 }(2002), 185--226.

\bibitem{Wahl} J.~Wahl. The Jacobian algebra of a graded
Gorenstein singularity. {\it Duke Math.\ J.\ }{\bf 55 }(1987),
843--871.

\end{thebibliography}
\end{document}